\documentclass[3p,times]{elsarticle}

\usepackage{ecrc}

\volume{00}

\firstpage{1}

\journalname{Computers \& Mathematics with Applications}

\runauth{}

\jid{procs}

\jnltitlelogo{\small Computers \& Mathematics with Applications}

\CopyrightLine{2014}{Published by Elsevier Ltd.}

\usepackage{multirow}
\usepackage{amssymb}
\usepackage{amsbsy}
\usepackage{amscd}
\usepackage{amsfonts}
\usepackage{amsmath}
\usepackage{amssymb}
\usepackage{amstext}
\usepackage{amsthm}
\usepackage{enumerate}
\usepackage{epsfig}
\usepackage{fancybox}
\usepackage{graphicx}
\usepackage{latexsym}
\usepackage{niceframe}
\usepackage{lipsum} 
\usepackage{tikz}
\usepackage{color}
\usepackage{dsfont}

\usepackage{amsthm}

\usepackage{lineno}  

\usepackage[figuresright]{rotating}

\newtheorem{theorem}{Theorem}

\newtheorem{remark}[theorem]{Remark}
\usepackage{hhline}

\begin{document}
\begin{frontmatter}



\dochead{}

\title{On high-order conservative finite element methods}


\author[EA]{Eduardo Abreu}
\author[EA]{Ciro D\'iaz}
\author[JG]{Juan Galvis}
\author[MS]{Marcus Sarkis}

\address[EA]{University of Campinas, Department of Applied Mathematics,
13.083-970, Campinas, SP, Brazil; eabreu@ime.unicamp.br}


\address[JG]{Departamento de Matem\'{a}ticas, Universidad Nacional de 
Colombia, Bogot\'{a} D.C., Colombia.}

\address[MS]{Department of Mathematical Sciences, Worcester Polytechnic 
Institute Worcester USA.}

\begin{abstract}
We describe and analyze a volumetric and residual-based Lagrange multipliers saddle
point reformulation of the standard high-order finite method, to impose conservation
of mass constraints for simulating the pressure equation on two dimensional convex
polygons, with sufficiently smooth solution and mobility phase. We establish high-order 
a priori error estimates with locally conservative fluxes and
numerical results are presented that confirm the theoretical results. 
\end{abstract}


\end{frontmatter}


\section{Problem}

Many porous media related practical problems lead to the 
numerical approximation of the pressure equation
\begin{eqnarray}\label{eq:problem1}
-\mbox{div}(\Lambda(x) \nabla p) = q \quad 
\mbox{in}~~\Omega \subset \Re^{2}, \\
p = 0 \quad \text{on}~~\partial \Omega_D, \\
\nabla p \cdot \mathbf{n} = 0  \quad \text{on}~~\partial \Omega \backslash \partial \Omega_D,
\end{eqnarray}
where $\partial \Omega_D$ is the part of the boundary
of the domain $\Omega$ (denoted by $\partial \Omega$) where the
Dirichlet boundary condition is imposed. In case the measure of
$\partial \Omega_D$ (denoted by $|\partial \Omega_D|$)
is zero, we assume the compatibility condition $\int_\Omega q \,dx= 0$.  On the above equation we have assumed 
without loss of generality homogeneous boundary conditions since we 
can always reduce the problem to that case.  The domain $\Omega$ is assumed to be a convex polygonal region in order at least $H^2$ regularity, see \cite{MR0775683}, and for a rectangle domain the problem is $H^p$ regular
for any integer $p$. We note however that this convexity or rectangularity 
are not required for the discretization, they are    
required only when regularity theory of partial differential equations (PDEs) 
is considered for establishing the a priori error estimates.

In multi-phase  immiscible 
incompressible flow, $p$ and $\Lambda$ are the unknown 
pressure and the given phase mobitity of one of the phases in 
consideration {(water, oil or gas); ( see e.g., 
\cite{PB14,BACC13,DFPY97,JDBR16,EA2014,ADFMP06}).} In general, the 
forcing term $q$ is due to gravity, phase transitions, sources and 
sinks, or when we transform a nonhomogeneous boundary condition problem
to a homogeneous one. The mobility phase in consideration is defined by   
$\Lambda(x) =  K(x) k_r(S(x))/\mu$, where $K(x)$ is 
the absolute (intrinsic) permeability of the porous media, $k_r$ is the relative 
phase permeability and $\mu$ the phase viscosity of the 
fluid. The assumptions required in this numerical analysis article may not in general hold for
such large-scale flow models. 

The main goal of our work is to obtain conservative solution  
of the equations above when they are discretized by high 
order continuous piecewise polynomial spaces. The obtained 
solution satisfies some given set of linear restrictions 
(may be related to subdomains of interest). Our motivations 
come from the fact that in some applications it is 
\underline{\bf imperative} to have some conservative 
properties represented as conservations of total flux in 
control volumes. For instance, if $\mathbf{q}^h$ represents the 
approximation to the flux (in our case 
$\mathbf{q}^h=-\Lambda \nabla p^h$ where $p^h$ is the 
approximation of the pressure), it is required that
\begin{equation} \label{conserv}
\int_{\partial V} \mathbf{q}^h \cdot 
\mathbf{n} = \int_{V} q \quad 
\text{for each control volume} ~~ V.
\end{equation}
Here $V$ is a control volume that does not cross 
$\partial \Omega_D$ from a set of controls volumes of 
interest, and here and after $\mathbf{n}$ is the normal 
vector pointing out the control volume in consideration. 
If some appropriate version of the total  flux restriction 
written above holds, the method that produces such an 
approximation is said to be a conservative discretization.

Several schemes offer conservative discrete solutions. 
These schemes depend on the formulation to be approximated 
numerically. Among the conservative discretizations for the 
second the order formulation the elliptic problem we mention
the finite volume (FV) method, some finite difference methods 
and some discontinuous Galerkin methods. On the other hand, 
for the first order formulation or the Darcy system we have 
the mixed finite element methods and some hybridizable 
discontinuous Galerkin (HDG) methods.
 
In this paper, we consider methods that discretize the second 
order formulation \eqref{eq:problem1}. Working with the 
second order formulation makes sense especially for cases where 
some form of high regularity holds. Usually in these cases 
the equality in the second order formulation is an equality 
in $L^2$ so that, in principle, there will be no need to 
weaken the equality by introducing less regular spaces for 
the pressure as it is done in mixed formulation with $L^2$ 
pressure. 

For second order elliptic problems, a very popular conservative discretization 
is the finite volume (FV) method. The classical FV discretization  provides and 
approximation of the solution in the space of piecewise linear 
functions with respect to a triangulation while satisfying 
conservation of mass on elements of a dual triangulation. 
When the approximation of the piecewise linear space is not 
enough for the problem at hand, advance approximation spaces 
need to be used (e.g., for problems with smooth solutions some 
high order approximation may be of interest). However, in 
some cases, this requires a sacrifice of the conservation 
properties of the FV method. Here in this paper, we  design and 
analyze conservative solution in spaces of high order 
piecewise polynomials. We follow the methodology  in \cite{MR3430146}, 
that imposes the total flux restrictions by employing Lagrange 
multiplier technique. This methodology was developed in order to apply the higher-order methods constructed in \cite{egh12,jcp,ge09_1,ge09_1reduceddim} to two-phase flow problems.

We note that FV methods that use higher degree piecewise 
polynomials have been introduced in the literature. The 
fact that the dimension of the approximation spaces is 
larger than the number of restrictions led the researchers 
to design some method to select solutions: For instance, 
in \cite{chen,chen2,chen3} to introduce additional control 
volumes to match the number of restrictions to the number 
of unknowns. It is also possible to consider a Petrov-Galerkin 
formulation with additional test functions rather that only 
piecewise constant functions on the dual grid. Other approaches.
 have been also introduced, see for instance 
\cite{CJimsfv} and references therein.

In the construction of new methodologies into 
a reservoir simulation should have into account the following 
issues: 1) local mass conservation properties, 2) stable-fast 
solver and 3) the flexibility of re-use of the novel technique 
into more complex models (such as to nonlinear time-dependent 
transport equations equation for the convection dominated 
transport equation). For Darcy-like model problems with 
very high contrasts in heterogeneity, the discretization 
of Darcy-like models alone may be very hard to solve 
numerically due to a large condition number of the arising 
stiffness matrix. Moreover, the situation in even more 
intricate for modeling non trivial two- \cite{LD93,MGJD05} 
and three-phase \cite{ADFMP06,EA2014} transport convection 
dominated phenomena problems for flow through porous media 
(see also other relevant works \cite{GYF15,PB14,JDBR16,JBAC10}). 
Thus, to achieve a sufficiently coupling between the volume 
fractions (or saturation) and the pressure-velocity, the full 
problem can be treated along with a fractional-step numerical 
procedure \cite{ADFMP06,EA2014}; we point out that we are 
aware about the very delicate issues linked to the discontinuous 
capillary-pressure (see \cite{BACC13} and the references therein).
Indeed, the fluxes (Darcy velocities) are smooth at the vertices 
of the cell defining the integration volume in the dual 
triangularization, since these vertices are located at the 
centers of non-staggered cells, away from the jump discontinuities 
along the edges. This facilitates the construction of second-order 
and high-order approximations linked to the hyperbolic-parabolic 
model problem \cite{ADFMP06,EA2014}. This gives some of the 
benefits of staggering between primal and dual mesh triangulation 
by combining our novel high-order conservative finite element 
method with finite volume for hyperbolic-parabolic conservation 
laws modeling fluid flow in porous media applications.

Here in this paper, we consider a  Ritz formulation and construct 
a solution procedure that combines a continuous Galerkin-type 
formulation that concurrently satisfies mass conservation 
restrictions. We impose finite volume restrictions by using 
a scalar Lagrange multiplier for each restriction. This is 
equivalently to a constraint minimization problem where we
minimize the energy functional of the equation restricted 
to the subspace of functions that satisfy the conservation 
of mass restrictions. Then, in the Ritz sense, the obtained  
 solution is the best among all functions that satisfy 
the mass conservation restriction.

Another advantage of our formulation is that the analysis 
can be carried out with classical tools for analyzing 
approximations to  saddle point problems 
\cite{MR2168342}. We analyze the method using an abstract 
framework and give an example for the case of second order 
piecewise polynomials. An important finding of these paper 
is that we were able to obtain optimal error estimates in 
the $H^1$ norm as well as the $L^2$ norm. Our $L^2$
error analysis requires additional assumptions, including specially collocated
dual meshes and $\Lambda = I$, and is obtained by adding the Lagrange multipliers to the
approximation $p_h$ by an Aubin-Nitsche trick  
\cite{MR2373954,MR2322235}.

The rest of the paper is organized as follows. In Section 
\ref{sec:con} we present the Lagrange multipliers formulation 
of our problem. In Section \ref{sec:dis} we introduce the 
saddle point approximation for which the analysis
is presented in Section \ref{sec:analysis}. In Section 
\ref{sec:smooth} we present the particular cases of 
high-order continuous finite element spaces. For this last 
case we present some numerical experiments in Section 
\ref{sec:num}. To close the paper we present some conclusions 
in Section \ref{sec:conclusions}.
 
\section{Lagrange multipliers and conservation of mass}\label{sec:con}

Denote  $H_{D}^1(\Omega)$ as the subspace of functions in $H^1(\Omega)$ 
which vanish
on $\partial \Omega_D$. In case $|\partial \Omega_D| = 0$, 
$H_{D}^1(\Omega)$ is the subspace of functions in $H^1(\Omega)$ with 
zero average on $\Omega$.
The variational formulation of problem  (\ref{eq:problem1}) is to find 
$p\in H^1_D(\Omega)$   such that 
\begin{equation}\label{eq:problem}
a(p,v)=F(v)   \quad \mbox{ for all } v\in H_D^1(\Omega),
\end{equation}
where the bilinear form $a$ is 
defined by
\begin{equation}\label{eq:def:a}
a(p,v)=\int_\Omega 
 \Lambda (x)\nabla p(x)\cdot \nabla v(x) dx, 
\end{equation}
and the functional $F$ is defined by 
\begin{equation}
{F(v)=\int_\Omega q(x)v(x)dx.  }
\end{equation}

In order to consider a general formulation for porous media applications
we let 
$\Lambda$ be a $2\times2$ matrix with entries in  $L^\infty   (\Omega)$
 in Problem (\ref{eq:problem}) to be almost everywhere symmetric positive definite 
matrix with eigenvalues bounded uniformly from below by
a positive constant, however, in certain 
parts of the paper when analysis and regularity theory are required, we assume 
$\Lambda(x)=I (identity)$.  The Problem (\ref{eq:problem}) is 
equivalent to the minimization problem: Find $p \in H_D^1$ and such that 
\begin{equation}\label{eq:unconstraint}
p  =\arg\min_{v \in H_D^1(\Omega) } \mathcal{J}(v),
\end{equation}
where
\begin{equation}\label{eq:def:J}
\mathcal{J}(v)=\frac{1}{2}a(v,v)-F(v).\end{equation} 

In order to deal with mass conservation properties we 
adopt the strategy introduced in \cite{MR3430146}. Let us
introduce the meshes we are going to use in our discrete 
problem. Let the primal triangulation ${\mathcal{T}}_h =\{  R_\ell \}_{\ell=1}^{N_h}$ be made of 
elements that are triangles or squares and let $N_h$ be  
the number of elements of this triangulation. We also 
have a dual mesh ${\mathcal{T}}_h^*=\{V_k \}_{k=1}^{N^*_h}$ 
where the elements are called control volumes, and  
 $N_h^*$ is the number of control volumes. Figure 
\ref{figure1} illustrates a primal and dual mesh made 
of squares when $\partial \Omega_D = \partial \Omega$, and 
in this case $N^*_h$ is equal to the number of interior
vertices of the primal triangulation. In general it is selected one control 
volume $V_k$ per vertex of the primal triangulation when
the measure $|V_k \cap \partial \Omega_D| = 0$.  In case 
$|\partial \Omega_D| = 0$, $N_h^*$ is the total number of vertices
of the  primal triangulation including the  vertices on $\partial \Omega$. 

In order to ensure the mass conservation, we impose it 
as a restriction (by using Lagrange multipliers) in 
each control volume $\{ V_{k} \}_{k=1}^{N^*_h}$.  We 
mention that our formulation allows for a more general 
case where only few control volumes, not necessarily related to 
the primal triangulation, are selected.

Let us define the linear functional 
$\tau_k(v) = \int_{\partial V_k} - \Lambda  \nabla v\cdot\mathbf{n}\,ds$,
$1 \leq k \leq N^*_h$. We first note $\tau_k(v)$ is not well defined
for $v \in H^1_D(\Omega)$. To fix that, recall that $q\in L^2(\Omega)$, 
therefore, let us define the Hilbert space
\[
{H^1_{div,\Lambda}(\Omega)} = \{ v: v \in H_D^1(\Omega) \mbox{ and }
\Lambda \nabla v \in 
\mbox{H}(\mbox{div},\Omega)\}
\]
with norm
$\|v\|_{H^1_{div,\Lambda}(\Omega)}^2 =
\|\Lambda \nabla v \cdot \nabla v\|_{L^2(\Omega)}^2 + \|\mbox{div}(\Lambda \nabla v)\|^2_{L^2(\Omega)}$
where the divergence is taken in the weak sense. We note
that this space and norm are well-defined with the properties of $\Lambda$
described above, that is, the smaller eigenvalue of $\Lambda(x)$ is
uniformly bounded from below by a positive number, by using
similar arguments given in ~\cite{MR0431752}*{Theorem 1}.
It is easy to see by using integration by parts with the function $z=1$
that $\tau_k$ is a continuous linear functional on ${H^1_{div,\Lambda}(\Omega)}$.
The integration by parts can be performed since 
$\|\Lambda \nabla v \cdot \nabla v\|_{L^2(V_k)}^2 + \|\mbox{div}(\Lambda \nabla v)\|^2_{L^2(V_k)}$ is well-defined and bounded by
$\|\Lambda \nabla v \cdot \nabla v\|_{L^2(\Omega)}^2 + \|\mbox{div}(\Lambda \nabla v)\|^2_{L^2(\Omega)}$. 

   Let $p$ be the solution of \eqref{eq:problem} and define
$m_k=\tau_k(p) = \int_{V_k} q\,ds$, $1\leq k\leq N_h^*$.
The problem (\ref{eq:unconstraint}) is also equivalent to: 
Find $p\in {H^1_{div,\Lambda}(\Omega)}$ such that 
\begin{equation}\label{eq:problem-with-restriction}
p =\arg\min_{v\in \mathcal{W}} \mathcal{J}(v),
\end{equation}
where 
\[
\mathcal{W}=\{ v: v \in {H^1_{div,\Lambda}(\Omega)}\mbox{ such that }  
\tau_k(v)=m_k, \quad 1\leq k\leq N^*_h\}.
\]
Problem (\ref{eq:problem-with-restriction}) above 
can be view as Lagrange multipliers min-max optimization 
problem. See \cite{MR2168342} and references therein. 
Then, in case an approximation of $p$, say $p_h$  is 
required to satisfy the constraints  $\tau_k(p^h)=m_k$, 
$1\leq k\leq N^*_h$, we can do that by discretizing directly the 
formulation \eqref{eq:problem-with-restriction}. In 
particular, we can apply this approach to a set of 
mass conservation restrictions used in finite volume 
discretizations.

In order to proceed with the associate Lagrange formulation, we define 
$M^h = \mathbb{Q}^0({\mathcal{T}}_h^*)$ to be the space of piecewise 
constant functions on the dual mesh ${\mathcal{T}}_h^*$. For $\mu\in M^h $,
depending on the context, we also interpret   $\mu$ as the vector  $[\mu_k]_{k=1}^{N_h^*}\in \mathbb{R}^{N_h^*}$ where 
$\mu_k=\mu|_{V_k}$. The 
Lagrange multiplier formulation of problem 
(\ref{eq:problem-with-restriction}) can be written as:
Find $p \in {H^1_{div,\Lambda}(\Omega)}$ and $\lambda \in M^h$ 
that solve:
\begin{equation}\label{eq:discrete-problem-with-restriction-lag}
\{p,\lambda\} = \arg
\max_{\mu\in M^h }\min_{v\in {H^1_{div,\Lambda}(\Omega)},} \mathcal{J}(v)-
(\overline{a}(p,\mu)-\overline{F}(\mu)).
\end{equation}
Here, the total flux bilinear form 
$\overline{a}: {H^1_{div,\Lambda}(\Omega)}\times\, M^h \to \mathbb{R}$ 
is defined by
\begin{equation}\label{eq:def:overline-a}
\overline{a}(v,\mu)=
\sum_{k=1}^{N^*_h}
\int_{\partial V_{k}} -\Lambda  \nabla v\cdot 
\,\mathbf{n}\,\mu=
\sum_{k=1}^{N^*_h}
\mu_k\int_{\partial V_{k}} -\Lambda  \nabla v\cdot 
\mathbf{n} \quad \mbox{ for all 
}  v\in {H^1_{div,\Lambda}(\Omega)} \mbox{ and } \mu\in M^h .
\end{equation}
The functional 
$\overline{F}:M^h  \to \mathbb{R}$ is defined 
by\[\overline{F}(\mu)=\sum_{i=k}^{N_h^*} 
\mu_k\int_{V_{k }}q\quad  \mbox{ for all } \mu \in  M^h .\]
Note that problem (\ref{eq:discrete-problem-with-restriction-lag}) depends on 
${\mathcal{T}}_h^*$ and therefore depends on $h$. The first order conditions of the min-max problem above  
give the following saddle point problem:
Find $p \in {H^1_{div,\Lambda}(\Omega)},$ and $\lambda\in M^h $ 
that solve:
\begin{equation}\label{eq:saddlepoint}
\begin{array}{llr}
  a(p,v)+\overline{a}(v,\lambda)&=F(v) &\mbox{ for all } v\in {H^1_{div,\Lambda}(\Omega)}, \\
\overline{a}(p,\mu) &=\overline{F}(\mu)& \mbox{ for all } \mu \in M^h.\\
\end{array}
\end{equation}
See for instance \cite{MR2168342}. Note that if the exact 
solution of problem (\ref{eq:unconstraint}) satisfies the 
restrictions in the saddle point formulation above we 
have $\lambda=0$ and we get the uncoupled system
\begin{equation}\label{eq:saddlepoint2}
\begin{array}{llr}
a(p,v)&=F(v) &\mbox{ for all } v\in H^1_{div,\Lambda}(\Omega), \\
\overline{a}(p,\mu) &=\overline{F}(\mu)& \mbox{ for all } \mu \in M^h .\\
\end{array}
\end{equation}
Also observe that the second equation above corresponds 
to a family of equations, one for each triangulation parametrized by 
$h$, all of them have the same solution. 

\section{Discretization}\label{sec:dis}

Recall that we have introduced a primal mesh 
${\mathcal{T}}_h =\{  R_\ell \}_{\ell=1}^{N_h}$ {made of 
elements that are triangles or squares}. We also 
have given a dual mesh $ {\mathcal{T}}_h^*=\{  V_k \}_{k=1}^{N^*_h}$ 
where the elements are called control volumes. In order to fix ideas we assume that the number of control volumes of  ${\mathcal{T}}_h^*$ equals the number of free vertices of  ${\mathcal{T}}_h^*$. Figure 
\ref{figure1} illustrates a primal and dual mesh made 
of squares for the case $\partial \Omega_D=\partial \Omega$.

\begin{figure}[ht]
\centering
\includegraphics[scale=.5]{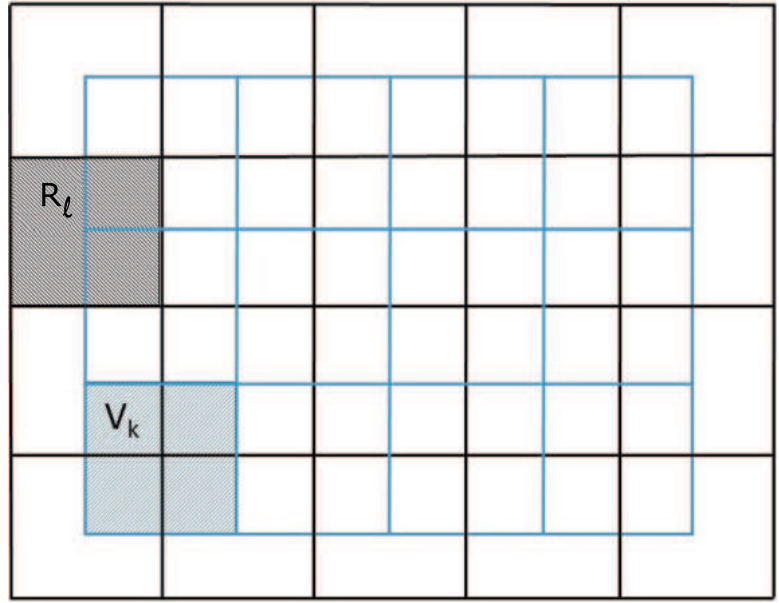}
\caption{Example of regular mesh made of squares and its dual 
mesh for the case $\partial \Omega_D=\partial \Omega$.}\label{figure1}
\end{figure}

Let us consider $ P^h = \mathbb{Q}^r ({\mathcal{T}}_h)$ the 
space of continuous and piecewise polynomials 
of degree $r$ on each element of the primal mesh, and 
$P^h_D = P^h \cap H_D^1(\Omega)$ (which are the functions 
in  $P^h$ that vanish in $ \partial \Omega_D $). Let 
$M^h = \mathbb{Q}^0({\mathcal{T}}_h^*)$ be the space of piecewise
constant functions on the dual mesh ${\mathcal{T}}_h^*$. We 
mention here that our analysis may be extended to 
different spaces and differential equations. See for 
instance \cite{MR3430146} where we consider GMsFEM 
spaces instead of piecewise polynomials. 

The discrete version of (\ref{eq:saddlepoint}) is to find 
$p^h \in P^h_D$ { and} $\lambda_h \in M^h$ 
such that

\begin{equation}\label{cela}
\begin{array}{llr}
a(p^h,v^h)+\overline{a}(v^h,\lambda^h)&=F(v^h) &\mbox{ for all } v^h\in P^h_D, \\
\overline{a}(p^h,\mu^h) &=\overline{F}(\mu^h)& \mbox{ for all } \mu^h \in M^h.\\
\end{array}
\end{equation}

Let $\left\lbrace \varphi_i \right\rbrace$ be the 
standard basis of $P^h_D$. We define the matrix
\begin{equation}
A = \left[ a_{i,j} \right] \quad \mbox{ where } 
a_{ij} = \int_{\Omega} \Lambda \nabla\varphi_i 
\cdot \nabla\varphi_j.
\end{equation}

Note that $A$ is the finite element stiffness matrix 
corresponding to finite element space $P^h_D$. Introduce 
also the matrix
\begin{equation}
\overline{A} = \left[ \overline{a}_{k,j} \right]  
\quad \mbox{ where } \overline{a}_{k,j} = 
\int_{\partial V_k} -\Lambda \nabla\varphi_j \cdot \textbf{n}.
\end{equation}

With this notation, the matrix form of the discrete
saddle point problem is given by,
\begin{equation}\label{eq:saddle}
\left[
\begin{array}{cc}
A&\overline{A}^T \\
\overline{A}&O
\end{array} \right]
\left[ \begin{array}{c}
p^h\\
\lambda^h
\end{array}
\right] = \left[ \begin{array}{c}
f\\
\overline{f}
\end{array}
\right]
\end{equation}
where the vectors $f=[f_i]_{i=1}^{N_h}$ and $\overline{f}=[\overline{f}_k]_{k=1}^{N_h^*}$  are defined respectively by
\[f_{i} = 
\int_{\Omega} q\, \varphi_i \quad \mbox{ and } \quad  \overline{f}_k= \int_{V_k} q.
\]

For instance, in the case of the primal and dual 
triangulation of Figure \ref{figure2} and polynomial degree $r=2$, 
the finite element matrix $ A $  is a sparse matrix 
with 19 diagonals. Also, for a control volume $ V_k $ there are at most $9$ supports 
of basis functions $\varphi_j$ with non-empty intersection with it, see Figure \ref{figure2}.

\begin{figure}[ht]
\centering
\includegraphics[scale=.6]{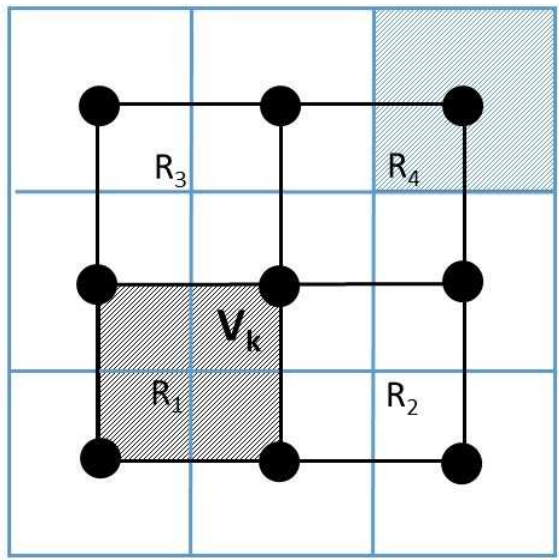}
\caption{Control volumes that intersect the support of 
a $ \mathbb{Q}^2 ({\mathcal{T}}_h)$ basis function.}\label{figure2}
\end{figure}

\begin{remark}
Note that matrix $\overline{A}$ is related to classical 
(low order) finite volume matrix. Matrix $\overline{A}$ 
is a rectangular matrix with more columns than rows. 
Several previous works on conservative high-order 
approximation of second order elliptic problem have 
been designed by ``adding'' rows using several constructions. For instance, one can proceed as follows:
\begin{enumerate}
\item Construct additional control volumes and test the 
approximation spaces against piecewise constant functions 
over the total of control volumes (that include the dual 
grid element plus the additional control volumes). We 
mention that constructing additional control volumes is 
not an easy task and might be computationally expensive.
We refer the interested reader to \cite{chen,chen2,chen3} 
for additional details.

\item Use additional basis functions that correspond to nodes
 other than vertices 
to obtain an FV/Galerkin formulation. This option has the 
advantage that no geometrical constructions have to be 
carried out. On the other hand, this formulation seems 
difficult to analyze. Also, some preliminary numerical 
tests suggest that the resulting linear  system becomes 
unstable for higher order approximation spaces (especially 
for the case of high-contrast multiscale coefficients).

\item Use the Ritz formulation with restrictions \eqref{cela}.

\end{enumerate}

Note that if  piecewise polynomials of degree $r=1$ are used, 
in the linear system \eqref{eq:saddle}, 
the restriction matrix corresponds to the usual finite volume 
matrix.  This matrix is known to be invertible. In this case, 
the affine space $\mathcal{W}$ is a singleton. Moreover, 
the only  function $p_h$ satisfying the restriction is given 
by $p_h=(\overline{A})^{-1} \overline{f}$. The Ritz formulation 
\eqref{cela} reduces to the classical finite volume method.

Then, in the Ritz sense, the solution of \eqref{cela} is 
not worse than any of the solutions obtained by the method 1.\ or
 2.\ mentioned above. Furthermore, the solution of the 
associated linear system \eqref{cela}, which is a saddle 
point linear system can be readily implemented using
efficient solvers for the matrix $A$ (or efficient solvers 
for the classical finite volume matrix $\overline{A}$);
See for instance \cite{MR2168342}. Additionally, we mention 
that the analysis of the method can be carried out using 
usual tools for the analysis of restricted minimization of 
energy functionals and mixed finite element methods. The 
numerical analysis of our methodology is under current 
investigation and it will presented elsewhere.
\end{remark}

\section{Analysis}\label{sec:analysis}

We show next that imposing the conservation in control 
volumes using Lagrange multipliers does not interfere 
with the optimality of the approximation in the $H^1$ 
norm. As we will see, imposing constraints will result 
in non optimal $L^2$ approximation but we were able to 
reformulate the $L^2$ approximation to get back to the 
optimal approximation by using the discrete Lagrange multiplier
as a corrector. 

Before proceeding 
we introduce notation to avoid proliferation of constants. 
We use the notation $A \preceq B$ to indicate that there 
is a constant $C_1$ such that $A\leq C_1 B$. If additionally 
there exist $C_2$ such that $B\leq C_2 A$ we write $A \asymp B$. These
constants do not depend on $\Lambda$, $u$, $u_h$, $\lambda_h$, $q$,
they might depend on the shape regularity of the elements and
the shape of $\Omega$. \\ 

 Denote  $\|v\|^2_a = \int_\Omega \Lambda \nabla v \cdot  \nabla v$ 
for all $v\in H^1_D(\Omega)$ and let us remind that
$H^1_{div,\Lambda} := \{v \in H^1_D(\Omega): \Lambda \nabla v \in 
\mbox{H}(\mbox{div},\Omega)\}$, and set $V^h = \mbox{Span}\{P_D^h, H^1_{div,\Lambda}\}$. 
We present a concrete example of the norms of $V^h$ and $M^h$ in the 
next section, see (\ref{eq:norm-for-P2}) and (\ref{eq:norm-for-Mh}), 
respectively.\\

{\bf Assumption A:} There exist norms $\|\cdot \|_{V^h}$ and $\|\cdot \|_{M^h}$ for 
 $V^h$ and $M^h$, respectively, such that 
\begin{enumerate} 
\item Augmented norm: $\|v\|_a \leq \|v\|_{V^h}$ forall $v \in V_h$. 
\item Continuity: there exists $\|\bar{a}\| \in \mathbb{R}$ such that 
\begin{equation}\label{continuity_abar}
|\bar{a}(v,\mu^h)| \leq \|\bar{a}\|\ \|v\|_{V^h} \|\mu^h\|_{M^h}~~\forall 
v \in V_h~~~\mbox{and}~~~\mu^h \in M^h.
\end{equation}\,\item Inf-Sup:  there exists $\alpha >0$ such that 
\begin{equation}\label{infsup}
\inf_{\mu^h\in M^{h}} \sup_{v^h \in P^h_D} \frac{\overline{a}(v^h,
\mu)}{\Vert v^h\Vert_a \, \Vert\mu^h\Vert_{M^h}}\geq
\alpha > 0.
\end{equation}
\end{enumerate}

\begin{remark}
The Inf-Sup condition above can be replaced by:  
there exists $\alpha >0$ such that 
\begin{equation}\label{infsup2}
\inf_{\mu^h\in M^{h}} \sup_{v^h \in P^h_D} \frac{\overline{a}(v^h,
\mu)}{\Vert v^h\Vert_{V^h} \, \Vert\mu^h\Vert_{M^h}}\geq
\alpha > 0.
\end{equation}
\end{remark}

We have the following result. Assume that $\{p,\lambda\}$ 
is the solution of (\ref{eq:saddlepoint})
and $\{p_h,\lambda_h\}$ the solution of 
(\ref{cela}). We have the following result. The proof 
uses classical approximation techniques for saddle point problems. 

\begin{theorem}\label{thm:t1}
Assume that ``Assumption A'' holds. Then, there exists a constant 
$C$ such that
\[
\Vert p-p^h\Vert_a  \leqslant 2
\left(1 + \frac{\Vert\overline{a}\Vert}{\alpha} \right) 
\inf_{v^h\in P^h_D} \Vert p-v^h\Vert_{V^h}.\]

\end{theorem}

\begin{proof} Note that in both problems, \eqref{eq:saddlepoint}  
and \eqref{cela}, $\mu$ belongs to the finite dimensional 
subspace $M^{h}$. Also, the exact solution of the Lagrange 
multiplier component of \eqref{eq:saddlepoint2} is $\lambda=0$.
Now we derive error estimates following classical saddle point 
approximation analysis. Define
\begin{align*}
W^h(q) := \left\lbrace v_h\in P^h_D  :  
\overline{a}(v^h,\mu) = \overline{F}(\mu) \,\,  
\mbox{ for all } \mu\in M^{h} \right\rbrace
\end{align*}
and
\begin{align*}
W^h := \left\lbrace v_h\in P^h_D:  
\overline{a}(v^h,\mu) = 0 \,\,  \mbox{ for all } \mu\in 
M^{h} \right\rbrace.
\end{align*}
First we prove 
\begin{equation}\label{eq:af1}
\Vert p-p^h\Vert_a \leq 2 \inf_{w^h\in W^h(q)}\Vert p-w^h\Vert_a.
\end{equation}

The inf-sup above in \eqref{infsup} implies that $W^h(q)$ 
(as well as $W^h$) is not empty. Take any $w^h  \in W^h(q)$ 
and solve for $z^h$ the problem,
\begin{equation}  \label{eq:aux1}
  a(v^h,z^h) = F(z^h) - a(w^h,z^h) \quad \mbox{ for all } z^h  \in W_h.
\end{equation}
Since $a$ is elliptic there exists a unique solution and therefore
 \begin{equation}
 p^h = v^h-w^h, 
 \end{equation}
where $p^h$ is the solution of \eqref{cela}. We have from 
\eqref{eq:saddlepoint2} and \eqref{cela} and using 
\eqref{eq:aux1} that
\begin{align*}
a(v^h,v^h) &= a(p^h-w^h,v^h)\\
 &= a(p^h,v^h) - a(w^h,v^h)\\
  &= {F(v^h)}- a(w^h,v^h)\\
 &= a(p,v^h) - a(w^h,v^h)\\
 &= a(p-w^h,v^h).
\end{align*}
Then, by using the ellipticity of $a$, we have
\begin{equation}
\Vert v^h\Vert^2_a = a(v^h,v^h)=a(p-w^h,v^h) \leqslant  \Vert p-
w^h\Vert_a\Vert v^h\Vert_a.
\end{equation}
Then
\begin{align*}
\Vert p - p^h\Vert_a &\leqslant \Vert p-w^h\Vert_a + \Vert w^h-p^h\Vert_a \\ 
 &\leqslant \Vert p - w^h\Vert_a + \Vert p-w^h\Vert_a = 2 \Vert p-w^h\Vert_a	
\end{align*} 
so that \eqref{eq:af1} holds true.  

  We now show that
\begin{align}\label{eq:af2}
\inf_{w^h\in W^h(q)} \Vert p-w^h\Vert_a \leqslant 
\left(1 + \frac{\Vert\overline{a}\Vert}{\alpha} \right) 
\inf_{v^h\in P^h_D} \Vert p-v^h\Vert_{V^h}
\end{align}

Take any $v^h \in W^h$. The inf-sup condition \eqref{infsup} 
implies that there exists a unique $z^h \in P^h_D$ such that
\begin{align*}
\overline{a}(z^h,\mu) = \overline{a}(p-v^h,\mu) 
&& \mbox{for all } \mu\in M^h.
\end{align*}
Then we have that $z^h \neq 0$, 
\begin{align*}
\frac{\overline{a}(z^h,\mu)}{\Vert 
z^h\Vert_a \Vert\mu\Vert_{M^h}} \geq \alpha
\end{align*}
and therefore
\begin{align*}
\Vert z^h\Vert_a &\leqslant 
\frac{1}{\alpha}\cdot \frac{\overline{a}(z^h,\mu)}
{\Vert\mu^h\Vert}=\frac{1}{\alpha}\cdot 
\frac{\overline{a}(p-v^h,\mu)}{\Vert\mu\Vert_{M^h}}\\
 &\leqslant \frac{1}{\alpha}\Vert 
\overline{a}\Vert \Vert p-v^h\Vert_{V^h}.
\end{align*}
Note that we have used the continuity of $\bar{a}$ in the extended norm 
$\| \cdot \|_{V^h}$. Put $w^h = z^h + v^h$ then
\begin{align*}
\overline{a}(w^h,\mu) &= \overline{a}(z^h,\mu) 
+ \overline{a}(v^h,\mu)\\
&= \overline{a}(p-v^h,\mu) + \overline{a}(v^h,\mu)\\
&= \overline{a}(p,\mu)\\
&= \overline{F}(\mu).
\end{align*}
Therefore we have that $w^h \in W_h(q)$. Moreover, 
\begin{align*}
\Vert p-w^h\Vert_a &\leqslant\Vert p-v^h\Vert_a
+\Vert v^h-w^h\Vert_a\\
&\leqslant\Vert p-v^h\Vert_a + \Vert z_h\Vert_a\\
&\leqslant \Vert p-v^h\Vert_a + 
\frac{\Vert \overline{a}\Vert}{\alpha}\, \Vert p-
v^h\Vert_{V^h}\\
&\leqslant \left(1+ \frac{\Vert 
\overline{a}\Vert}{\alpha} \right) \Vert p-v^h\Vert_{V^h}.
\end{align*}

Combining \eqref{eq:af1} and \eqref{eq:af2} we get the result.
\end{proof}

From now on we assume from that $\Lambda = I$ (identity). In this case, 
$\|\cdot\|_a = |\cdot|_{H^1(\Omega)}$, and as we will see in Section \ref{sec:smooth} for
regular meshes and $\mathbb{Q}^r({\mathcal{T}}_h)$ elements that the ``Assumption A'' holds with
$1/\alpha = O(1)$, $|\bar{a}| = O(1)$ with the norms $V^h$ and $M^h$
defined in (\ref{eq:norm-for-P2}) and (\ref{eq:norm-for-Mh}), 
respectively. The next two Assumptions are discussed at the end of
Section \ref{sec:smooth}. 

{\bf Assumption B:} Assume that solution $p$  of the problem
(\ref{eq:problem1}) is in $H^{r+1}(\Omega)$ and the following
approximation holds for some integer $r \leq 1$
\[
\inf_{v_h \in P^h_D} \|p - v_h\|_{V^h} \preceq  h^r|p|_{H^{r+1}(\Omega)}.
\]

As a corollary of ``Assumptions A and B'' and Lemma \ref{thm:t1}, we obtain
\[
\Vert p-p^h\Vert_{V^h} \preceq h^r |p|_{H^{r+1}(\Omega)}.
\]

As we will show in the numerical experiments, the error 
$\Vert p-p^h\Vert_{L ^2(D)}$ is not optimal but according 
to the next result if we correct $p^h$ to $p^h+\lambda^h$ 
we recover the optimal approximation. The proof of the following 
results follows from a duality argument similar to that of the 
Aubin-Nitsche method; see \cite{MR2373954,MR2322235}. Let
us introduce the following regularity assumption:\\

{\bf Assumption C:} The problem is $H^2(\Omega)$ regular (see \cite{MR2373954}) if for any
$\tilde{q} \in L^2(\Omega)$ as a right-hand side for the problem
(\ref{eq:problem1}), its solution $\tilde{p}$ satisfies
\[
\|\tilde{p}\|_{H^2(\Omega)}\preceq \|\tilde{q}\|_{L^2(\Omega)}.
\]

\begin{theorem} \label{Theorem1}
Assume that $\Lambda = I$. Assume also that ``Assumptions A, B and C'' hold. 
Then, 
\[
\Vert p-(p^h+\lambda^h)\Vert_{L^2(\Omega)} \preceq h^{r+1} |p|_{H^{r+1}(\Omega)}.
\]
\end{theorem}
\begin{proof} For $g\in L^2$ define $\mathcal{S}^h_1g$ 
and $\mathcal{S}^h_0g$ as the solution of
\begin{align}\label{celaproof}
&a(\mathcal{S}^h_1g,v^h) + \overline{a}(v^h,\mathcal{S}^h_0g) = \int_D g v^h && 
\mbox{ for all } v^h \in H^1_{div,I} \\
&\overline{a}(\mathcal{S}^h_1g,\mu) = \int_D g v^h\mu  && \mbox{ for 
all } \mu \in M^h.
\end{align}
Analogously, define $Sg$ as the solution of

\begin{equation}\label{eq:saddlepoint2proof}
\begin{array}{llr}
a(Sg,v)&=\int_D gv&\mbox{ for all } v\in H^1_{div,I}, \\
\overline{a}(Sg,\mu^h) &=\int_D g \mu^h& \mbox{ for all } \mu^h \in M^h .\\
\end{array}
\end{equation}
Observe that $p^h=S^h_1q$, $\lambda ^h=S^h_0 q$ and $p=Sq$. 
According to our previous result in Theorem \ref{thm:t1} 
combined with standard regularity and approximation results 
(\cite{MR2373954}) we have
\begin{equation}\label{eq:auxaprox}
 ||Sg- S^h_1g ||_a\preceq \inf_{v^h\in P^h_D} ||Sg- v^h ||_{V^h} \preceq 
 h ||Sg||_{H^2(\Omega)} \preceq h ||g||_{L^2(\Omega)}.
\end{equation}

Recall that, 
\begin{align}\label{eq:auxAubintrick1}
\Vert p-(p^h+\lambda^h)\Vert_{L^2(\Omega)} =
\sup_{g\in L^2} \frac{(p-(p^h+\lambda^h), g)}{\Vert g\Vert_{L^2(\Omega)}}.
\end{align}
By using the definition of $S$, $S^h_0$ and $S^h_1$ 
in \eqref{celaproof} and  \eqref{eq:saddlepoint2proof} we get
\begin{align*}
&(p-(p^h+\lambda^h), g)
= (p, g)_{0}-(p^h, g)_{0}-(\lambda^h, g)_{0}\\
&= a( Sg,p)-\Big( a(S^h_1g, p^h)_{0}+\overline{a}(p^h, S_0^hg)\Big)
-\overline{a}(S_1^hg,\lambda^h) \\
&= a( Sg,p)-\Big( a(S^h_1g, p^h)_{0}+\overline{a}(S_1^hg,\lambda^h)\Big)
-\overline{a}(p^h, S_0^hg) \\
&= a( Sg,p)-\left( \int_D f S^h_1g \right)-\overline{a}(p^h, S_0^hg) \\
&= a( Sg,p)- a(p, S^h_1g )-\overline{a}(p^h, S_0^hg) \\
&= a(p,Sg- S^h_1g )-\overline{a}(p^h, S_0^hg) \\
&= a(p-p^h,Sg- S^h_1g )+a(p^h,Sg- S^h_1g )-\overline{a}(p^h, S_0^hg) \\
&= a(p-p^h,Sg- S^h_1g )+a(p^h,Sg) -\Big( a( S^h_1g,p^h )-\overline{a}(p^h,
 S_0^hg) \Big) \\
&= a(p-p^h,Sg- S^h_1g )+\int_D g p^h -\left(\int_D g p^h  \right) \\
&= a(p-p^h,Sg- S^h_1g )\\
&\leq  ||p-p^h||_a ||Sg- S^h_1g ||_a\\
 &\preceq h |p-p^h|_a ||g||_{L^2(\Omega)} 
\end{align*}
In the  last step we have used \eqref{eq:auxaprox}. Replacing 
the last inequality in \eqref{eq:auxAubintrick1} and with ``Assumption B'',
we get the  result.
\end{proof}

\section{The case of piecewise polynomials of degree two in regular meshes}
\label{sec:smooth}

In this section we consider a regular mesh made of squares. 
See Figure \ref{figure1}. 
Define \[
\Gamma^*_h=\bigcup_{k=1}^{N^*_h} \partial V_k =\bigcup_{k, k^\prime=1}^{N^*_h}
(\partial V_k \cap \partial V_{k^\prime} )
\] 
that is, $\Gamma^*_h$ is the interior interface generated by 
the dual mesh.  For $\mu\in M^h$ define $[\mu]$ on 
$\Gamma^*_h$ as the jump across element interfaces, that is, 
$[\mu]|_{\partial V_k \cap \partial V_{k^{\prime}}}= \mu_k-\mu_{k^{\prime}}$.
Note that for $p \in V^h$  
\begin{align*}
\overline{a}(p,{\mu})&=\sum_{k=1}^{N^*_h}
\mu_k\int_{\partial V_k } -\nabla p \cdot \mathbf{n} = 
\int_{\Gamma^*_h}-\nabla p \cdot \mathbf{n} \,\left[ \mu\right].
\end{align*}

For each control volume $V_k$, denote by $E(k)$ the set 
of element of the primal mesh that intersect $V_k$. Note 
that in each control volume we have 
\[
\int_{\partial V_k}-\nabla p \cdot \mathbf{n} = \sum_{\ell \in E(k) } 
\int_{\partial V_k\cap R_\ell }-\nabla p \cdot \mathbf{n}.
\]
To motivate the definition of the norms we study the 
continuity of the bilinear form $\overline{a}$. Observe 
that, 
\begin{align*}
\left(\int_{\partial V_k \cap \partial V_{k^\prime}}
-\nabla p \cdot \mathbf{n} \,\left[ 
{\mu}\right]\right)^2 &\leq 
\left( h\int_{\partial V_k \cap \partial V_{k^\prime}}
(\nabla p \cdot \mathbf{n} )^2 \right) \, \left( \frac{1}{h} 
\int_{\partial V_k \cap \partial V_{k^\prime}} [\mu]^2 \right).
\end{align*}
And therefore by applying Cauchy inequality and adding up we get,
\[
|\overline{a}(p,{\mu})|\leq 
\left( h\int_{\Gamma^*_h} (\nabla p \cdot \mathbf{n} 
)^2 \right)^{1/2}\left( \frac{1}{h} 
\int_{\Gamma^*_h} [\mu]^2 \right) ^{1/2}.
\]
Using a trace inequality we get that 
\begin{align}
h\int_{\Gamma^*_h} (\nabla p \cdot \mathbf{n} )^2& = h\sum_{k=1}^{N^*_h}
\int_{\partial V_k} (\nabla p \cdot \mathbf{n} )^2\\
&=\sum_{k=1}^{N^*_h}
\sum_{\ell \in E(k)}h\int_{\partial V_k\cap R_\ell} (\nabla p \cdot \mathbf{n} )^2\\
& \preceq \sum_{k=1}^{N^*_h}
\sum_{\ell \in E(k)} \left( |p|^2_{H^1(V_k\cap R_\ell)} + h^2 (\|p_{xx}\|^2_{L^2(V_k\cap R_\ell)} + \|p_{yy}\|^2_{L^2(V_k\cap R_\ell)}) \right) \\
& =\sum_{\ell=1}^{N_h}\left(  |p|_{H^1(R_\ell)}^2 + h^2 (\|p_{xx}\|^2_{L^2(R_\ell)} + \|p_{yy}\|^2_{L^2(R_\ell)}) \right)\\
& = |p|_{H^1(\Omega)}^2 +h^2 \sum_{\ell=1}^{N_h} (\|p_{xx}\|^2_{L^2(R_\ell)} + \|p_{yy}\|^2_{L^2(R_\ell)}).
\end{align}

Now we are ready to define the norm
\begin{equation}\label{eq:norm-for-P2}
\Vert p \Vert^2_{V^h} = \vert p\vert^2_{H^1(\Omega)} + h^2\sum_{\ell=1}^{N_h} 
 (\|p_{xx}\|^2_{L^2( R_\ell)} + \|p_{yy}\|^2_{L^2(R_\ell)})
\end{equation} 
Note that if $p \in \mathbb{Q}^1({\mathcal{T}}_h)$ then  
$\Vert p\Vert_{V^h}^2 = \vert p\vert^2_{H^1(\Omega)}$. Also, 
if $p\in \mathbb{Q}^2({\mathcal{T}}_h)$ we have 
$\Vert p\Vert_{V^h}^2 \leq c\vert p\vert^2_{H^1(\Omega)}$ by using
inverse inequality. 

Also define the discrete  norm for the spaces of Lagrange 
multipliers as
\begin{equation}\label{eq:norm-for-Mh}
\Vert  {\mu} \Vert_{M^h}^2 =\frac{1}{h}\int_{\Gamma^*_h}[ \mu ]^2.
\end{equation}

We have shown above that the form $\overline{a}$ is continuous, 
that is, there is a constant $|\bar{a}|$ such that, 
\begin{align*}
|  \overline{a}(p, {\mu})|\leq |\bar{a} \Vert p\Vert_{V^h} \Vert   {\mu}\Vert_{M^h}. 
\end{align*}
This also implies continuity in the $H^1$ norm. Now let us
show the inf-sup condition.

\begin{theorem} Consider the norms for $\|\cdot\|_{a} = |\cdot|_{H^1(\Omega)}$
and  $M^h$ defined in (\ref{eq:norm-for-Mh}), 
respectively. There is a constant  $\alpha$ such that, 
\begin{equation}\label{infsup3}
\inf_{\mu\in M^h} \sup_{v^h \in \mathbb{Q}^1({\mathcal{T}}_h)} \frac{\overline{a}(v^h,
\mu)}{\Vert v^h\Vert_a \, \Vert\mu\Vert_{M^h}}\geq
\alpha > 0.
\end{equation}
\end{theorem}
\begin{proof}
Given $ {\mu} \in M^h$ define $v\in \mathbb{Q}^1({\mathcal{T}}_h)$ as $v(x_i) = 
\overline{\mu}(x_i)$ if $x_i$ is a vertex of the primal mesh 
in $V_i$ and $v(x_i) = 0$ if  $x_i$ is a vertex of the primal mesh 
on $\partial \Omega_D$. We first 
verify that, 
\begin{equation}\label{eq:vforinfsup}
\vert v  \vert_{H^1}^2 = \Vert v\Vert_{V^h}^2 \asymp  
\Vert  {\mu }\Vert_{M^h}^2.
\end{equation}

It is enough to verify this equivalence of norms in the 
reference square $\hat{R}=[0,1]\times[0,1]$. Denote by $P_i$, 
$i=1,2,3,4$ the values of the reference function $\hat{v}$ at the nodes 
of the reference element. We have, 
\[
\hat{v}    = P_1(1-x)(1-y) + P_2(x)(1-y) + P_3(1-x)y +P_4 x y, \\
\]
and we can directly compute $\partial_x \hat{v}  = (P_2-P_1)(1-y) + (P_4-P_3)y$ 
and $ \partial_y \hat{v} =  (P_3-P_1)(1-x) + (P_4-P_2)x$. Therefore,
after some calculations we obtain 
\begin{align*} (P_2-P_1)^2\frac{1}{6} + (P_4-P_3)^2\frac{1}{6}
  &\leq(P_2-P_1)^2\frac{1}{3} + 
  (P_4-P_3)^2\frac{1}{3} - |(P_2-P_1)(P_4-P_3)|\frac{1}{3} \\
  &\leq \int_{\hat{R}} (\partial_x \hat{v})^2 \\
  & \leq
  (P_2-P_1)^2\frac{1}{3} + 
(P_4-P_3)^2\frac{1}{3} + |(P_2-P_1)(P_4-P_3)|\frac{1}{3}\\
             &\leq  (P_2-P_1)^2\frac{1}{2} + (P_4-P_3)^2\frac{1}{2}.                   
\end{align*}
 Analogously, 
\begin{align*} 
(P_3-P_1)^2\frac{1}{6} 
  + (P_4-P_2)^2\frac{1}{6} \leq \int_{R} (\partial_y
  \hat{v})^2 
             \leq (P_3-P_1)^2\frac{1}{2} + (P_4-P_2)^2\frac{1}{2}.                    
\end{align*}
This prove \eqref{eq:vforinfsup}. Now we verify that 
\begin{equation*}
\int_{\Gamma^*_h}\nabla v \cdot 
\mathbf{n} [ {\mu} ] \succeq \Vert \mu \Vert_{M^h}^2.
\end{equation*}
Observe that if $R$ is an element of the primal triangulation, 
$\Gamma^*_h \cap R$ can be written as the union of four segments denoted by 
$\Gamma^*_{i,R}$ where $i=4(up), 2(left), 3(right), 1(down)$. Working again
on the reference square, we have 
\begin{align*}
\int_{\hat{\Gamma}^*_{1,\hat{R}}}\nabla \hat{v}\cdot \mathbf{n} 
[P_2-P_1] &= (P_2-P_1)\int_0^{1/2}(P_2-P_1)(1-y) +(P_4-P_3)y\\
 &=  (P_2-P_1)^2\frac{3}{8} + (P_2-P_1)(P_4-P_3)\frac{1}{8}.
\end{align*}
Analogously, 
\begin{align*}
\int_{\hat{\Gamma}^*_{2,\hat{R}}}\nabla \hat{v}\cdot \mathbf{n}[P_4-P_3] 
         &=  (P_4-P_3)^2\frac{3}{8} + (P_2-P_1)(P_4-P_3)\frac{1}{8}\\       
\int_{\hat{\Gamma}^*_{3,\hat{R}}}\nabla \hat{v}\cdot \mathbf{n} [P_3-P_1] 
         &=  (P_3-P_1)^2\frac{3}{8} + (P_3-P_1)(P_4-P_2)\frac{1}{8}\\   
\int_{\hat{\Gamma}^*_{4,\hat{R}}}\nabla \hat{v}\cdot \mathbf{n} [P_4-P_2] 
         &=  (P_4-P_2)^2\frac{3}{8} + (P_3-P_1)(P_4-P_2)\frac{1}{8}\\                     
\end{align*}

If we add these last form equations we get 
\begin{align*}
\int_{\Gamma^*_h \cap R}\nabla v \cdot \mathbf{n} [ {\mu} ]  
\succeq (P_2-P_1)^2 +	 (P_3-P_1)^2 + (P_4-P_2)^2 + (P_4-P_3)^2.
\end{align*}
This finish our proof.
\end{proof}

We mention that for quasi-uniform and shape regular meshes, for quadrilateral
$\mathbb{Q}^r({\mathcal{T}}_h)$ or triangular $\mathbb{P}^r({\mathcal{T}}_h)$ finite element spaces, the ``Assumption B''
holds for $p \in H^{r+1}(\Omega)$. For the solution $p$ of problem
(\ref{eq:problem1}) to be in $H^{r+1}(\Omega)$, it is necessary to
impose conditions on the shape and smoothness of domain as well as on the
type of boundary conditions (Dirichlet, Neumann or mixed);
see \cite{MR0775683}. For instance, for the pure homogeneous Dirichlet boundary condition case, it is sufficient that $\Omega$ be convex and
$q \in L^2(\Omega)$ in order that $p \in H^{2}(\Omega)$,
and also ``Assumption C'' follows. 
For  $p$ to be in $H^{r+1}(\Omega)$ for integer $r\leq 2$, it is sufficient
that $\Omega$ be a rectangular domain and $q \in  H^{r-1}(\Omega)$. Higher-order approximation and regularity can also be obtained for curved
isoparametric finite elements on domains with smooth boundaries.

\section{Numerical Experiments}\label{sec:num}

We consider the Dirichlet problem \eqref{eq:problem1} and
employ the meshes depicted in Figure \ref{figure1} 
with a variety of mesh sizes and $\Lambda=I$. We impose conservation 
of mass as described in the paper by using Lagrange 
multipliers. For this paper, we solved the saddle 
point linear system by LU decomposition. Several 
iterative solvers can be proposed for this saddle 
point problem but this will be considered in future 
studies, not here.

Consider  \textbf{$\Omega = [0,1]\times [0,1]$} and $\Lambda=I$. We consider 
a regular mesh made of $2^M \times 2^M$ squares. The dual mesh is 
constructed by joining the centers of the elements of the 
primal mesh. We performed a series of numerical experiments 
to compare properties of FEM solutions with the solution 
of our high order FV formulation (to which we refer from 
now on as FV solution). The FV formulation with correction we denote
by  FV + $\lambda$. 

\subsection{Smooth problem with nonhomogeneous Dirichlet boundary conditions}
We selected the following forcing term and Dirichlet boundary conditions as
\begin{equation*}
\begin{aligned}
q(x,y) &= 2\pi (\cos(\pi x) \sin(\pi y) 
- 3\sin(\pi x) \cos(\pi y) 
+ \pi \sin(\pi x) \sin(\pi y) (-x+3y)), \\
u_D(x,y) &= 1+ x+ 2y,
\end{aligned}
\end{equation*}
and see that the exact solution is
\[
p(x,y) = \sin(\pi x) \sin(\pi y) (-x+3y) + 1 +x+2y.
\]
First we implemented the case of $\mathbb{Q}^1({\mathcal{T}}_h)$ elements that 
corresponds to the classical finite element and classical 
finite volume methods. We compute $L^2$ and $H^1$ errors. 
We present the results in Table \ref{TabL22Q1} and displayed 
graphically in Figures \ref{logL22Q1} and \ref{logH12Q1}. We 
observe here optimal convergence of both strategies.

\begin{figure}[ht]
\centering 
\includegraphics[scale=.3, trim = 4cm 0 0 0]{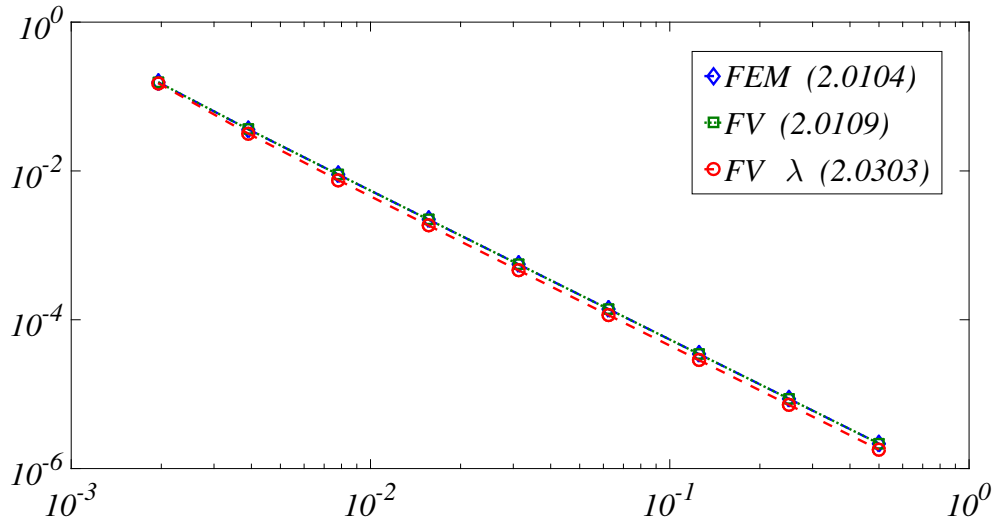}
\caption{Log-log graphic  of \textbf{FEM} and \textbf{FV} 
$L^2$ errors for  numerical solutions of Example 1, using 
$\mathbb{Q}^1({\mathcal{T}}_h)$ discretization, $h = 2^{-M}$, $M=1,\dots, 9$.}
\label{logL22Q1} 
\end{figure}

\begin{figure}[ht]
\centering
\includegraphics[scale=.3, trim = 4cm 0 0 0]{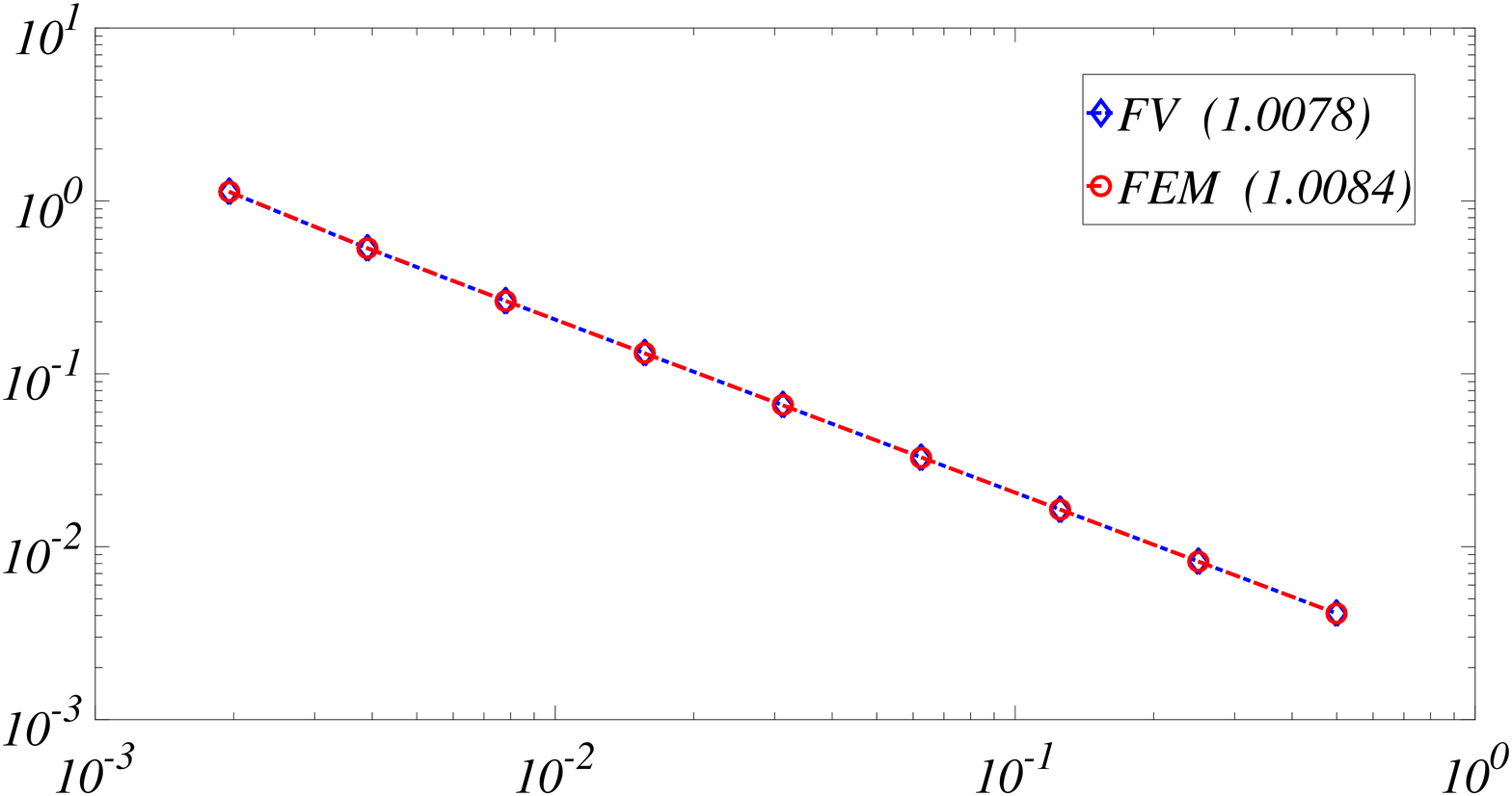} 
\caption{Log-log graphic  of \textbf{FEM} and \textbf{FV} 
$H^1$ errors for  numerical solutions of Example 1, using 
$\mathbb{Q}^1({\mathcal{T}}_h)$ discretization, $h = 2^{-M}$, $M=1,\dots, 9$.}
\label{logH12Q1}
\end{figure}

\begin{table}[ht]
\centering
\renewcommand{\arraystretch}{1.1}
      \begin{tabular}{c|c|c||c|c} 
       	    $M$      & $FEM, \,\, L^2 \, Error$ & $FV+\lambda, \, L^2 \, Error$ & $FEM, \,\, H^1 \, Error$ & $FV,  \, H^1 \, Error$\\ 
			\hline  $1$  &  $1.5538\times 10^{-1}$  &  $1.5103\times 10^{-1}$ &  $1.1297\times 10^{0}$  &  $1.1338\times 10^{0}$ \\
			\hline  $2$  &  $3.6342\times 10^{-2}$  &  $3.1881\times 10^{-2}$ &  $5.3226\times 10^{-1}$  &  $5.3416\times 10^{-1}$ \\
			\hline  $3$  &  $8.9720\times 10^{-3}$  &  $7.5.276\times 10^{-3}$ & $2.6374\times 10^{-1}$  &  $2.6403\times 10^{-1}$ \\
		 	\hline  $4$  &  $2.2548\times 10^{-3}$  &  $1.9348\times 10^{-3}$  &  $1.3163\times 10^{-1}$  &  $1.3172\times 10^{-1}$\\
		 	\hline  $5$  &  $5.5513\times 10^{-4}$  &  $4.6095\times 10^{-4}$  &  $6.5833\times 10^{-2}$  &  $6.5840\times 10^{-2}$ \\
            \hline  $6$  &  $1.3875\times 10^{-4}$  &  $1.1513\times 10^{-4}$  &  $3.2948\times 10^{-2}$  &  $3.2924\times 10^{-2}$\\
            \hline  $7$  &  $3.4685\times 10^{-5}$  &  $2.8776\times 10^{-5}$  &  $1.6418\times 10^{-2}$  &  $1.6489\times 10^{-2}$\\
            \hline  $8$  &  $8.6711\times 10^{-6}$  &  $7.1935\times 10^{-6}$  &  $8.2838\times 10^{-3}$  &  $8.2141\times 10^{-3}$\\
            \hline  $9$  &  $2.1678\times 10^{-6}$  &  $1.7983\times 10^{-6}$  &  $4.1639\times 10^{-3}$  &  $4.1857\times 10^{-3}$\\
        \end{tabular}
\caption{Table  of \textbf{FEM} and \textbf{FV} $L^2$ and 
$H^1$ errors for  numerical solutions of Example 1, using 
$\mathbb{Q}^1({\mathcal{T}}_h)$ discretization, calculated  over $9$ different values 
of mesh norm, $h = 2^{-M}$.}
\label{TabL22Q1}
\end{table}

We now consider the case of $\mathbb{Q}^2({\mathcal{T}}_h)$ finite element space. 
We have computed the   FEM solution as well 
as the solution of the saddle point system \eqref{cela}. 
We call this last solution the High order FV solution.
We estimate the $L^2$ and $H^1$ errors for both 
{FEM} and {FV} and compare the results 
through the log-log graphics shown in Figure \ref{logL22} 
and Figure \ref{logH12}. See also the  
Table \ref{TabL22} for comparisons. Numerical 
convergence is observed with a rate of $2$ for the $H^1$ 
error.  The error $p-p^h$ is not optimal in $L^2$. For 
this error, the observed convergence rate is close to 
$2$ but if we observe the error $p-(p^h+\lambda^h)$ 
in $L^2$ we estimate a convergence rate of $3$. 
{ These results coincide with our theoretical 
predictions four our High order FV formulation.}

\begin{figure}[ht]
\centering 
\includegraphics[scale=.3, trim = 4cm 0 0 0]{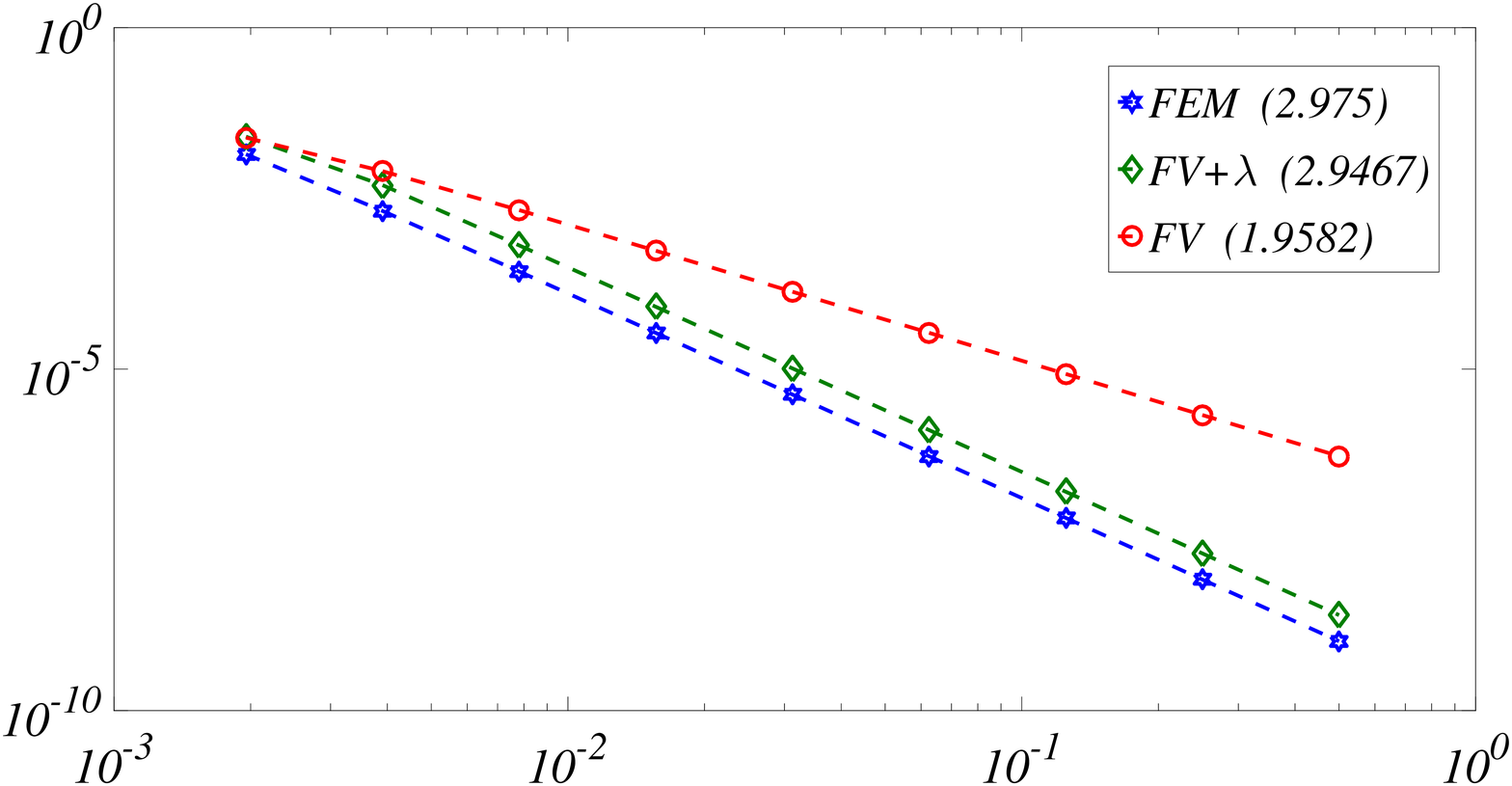}
\caption{Log-log graphic  of \textbf{FEM} and \textbf{FV} 
$L^2$ errors for  numerical solutions of Example 1, using 
$\mathbb{Q}^2({\mathcal{T}}_h)$ discretization,  $h = 2^{-M}$, $M=1,\dots, 9$.}
\label{logL22} 
\end{figure}
 
\newpage 
 
\begin{figure}[ht]
\centering
\includegraphics[scale=.3, trim = 4cm 0 0 0]{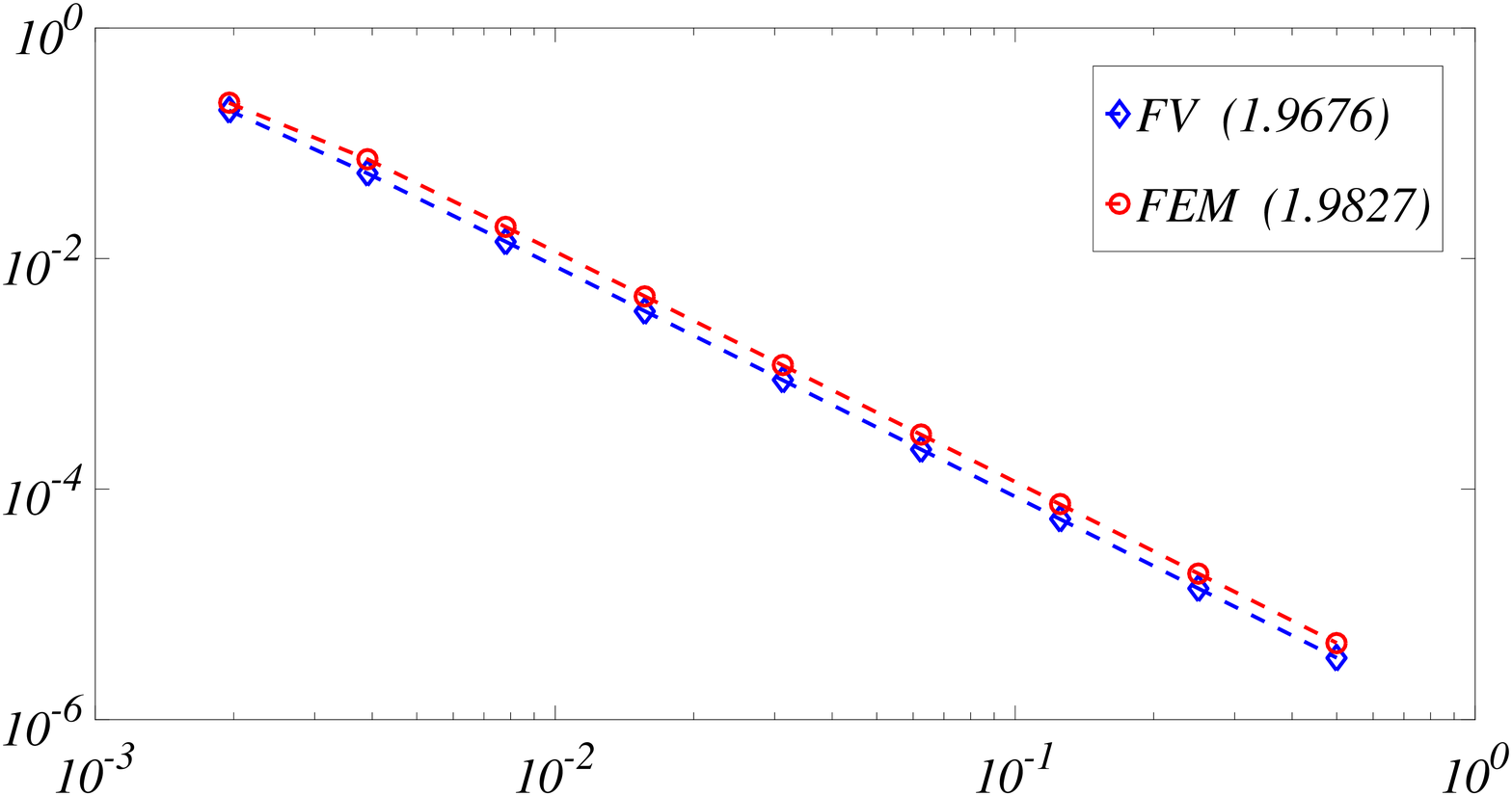} 
\caption{Log-log graphic  of \textbf{FEM} and \textbf{FV} $H^1$ 
errors for  numerical solutions of Example 1, using $\mathbb{Q}^2({\mathcal{T}}_h)$ 
discretization, $h = 2^{-M}$, $M=1,\dots, 9$.}
\label{logH12}
\end{figure}

\begin{table}[ht]
\centering
\centering
\renewcommand{\arraystretch}{1.1}
     	\begin{tabular}{c|c|c||c|c} 
       	    $M$      & $FEM \, L^2 \, Error$ & $FV+\lambda, \, L^2 \, Error$ & $FEM \, H^1 \, Error$ & $FV, \, H^1\,  Error$ \\
			\hline  $1$  &  $1.4061\times 10^{-2}$  &  $2.5448\times 10^{-2}$ &  $1.9302\times 10^{-1}$  &  $2.2436\times 10^{-1}$ \\
			\hline  $2$  &  $2.1217\times 10^{-3}$  &  $4.9023\times 10^{-3}$ &  $5.4862\times 10^{-2}$  &  $7.2895\times 10^{-2}$ \\
			\hline $3$   &  $2.6860\times 10^{-4}$  &  $6.4789\times 10^{-4}$ &  $1.4072\times 10^{-2}$  &  $1.8847\times 10^{-2}$  \\
		 	\hline  $4$  &  $3.3875\times 10^{-5}$  &  $8.1756\times 10^{-5}$ &  $3.5418\times 10^{-3}$  &  $4.7552\times 10^{-3}$ \\
		 	\hline  $5$  &  $4.2437\times 10^{-6}$  &  $1.0242\times 10^{-5}$ &  $8.3539\times 10^{-4}$  &  $1.2667\times 10^{-3}$  \\
            \hline  $6$  &  $5.3075\times 10^{-7}$  &  $1.2810\times 10^{-6}$ &  $2.2016\times 10^{-4}$  &  $2.9616\times 10^{-4}$ \\
            \hline  $7$  &  $6.6353\times 10^{-8}$  &  $1.6015\times 10^{-7}$ &  $5.5043\times 10^{-5}$  &  $7.4046\times 10^{-5}$ \\
            \hline  $8$  &  $8.2944\times 10^{-9}$  &  $2.0019\times 10^{-8}$ &  $1.3761\times 10^{-5}$  &  $1.8512\times 10^{-5}$ \\
            \hline  $9$  &  $1.0369\times 10^{-9}$  &  $2.5024\times 10^{-9}$ &  $3.4403\times 10^{-6}$  &  $4.6280\times 10^{-6}$ \\
	    \end{tabular}
\caption{Table  of \textbf{FEM} and \textbf{FV} $L^2$ and 
$H^1$ errors for  numerical solutions of Example 1, using 
$\mathbb{Q}^2({\mathcal{T}}_h)$ discretization, calculated  over $9$ different 
values of mesh norm, $h = 2^{-M}$.} 
\label{TabL22}
\end{table}

We now turn our attention to the norm $\| \cdot \|_{V^h}$, 
defined in \eqref{eq:norm-for-P2}, of the computed error. 
We introduce the  seminorm,
 \begin{equation}\label{def:seminormVh}
| p |_{V^h}^2 = \sum_{\ell=1}^{N_h} \left(
    \|p_{xx}\|_{L^2{(R_\ell)}}^2 +   \|p_{yy}\|_{L^2{(R_\ell)}}^2 \right)
\end{equation} 
Note that $\| p \|_{V^h}^2= |p|_{H^1}^2+h^2| p |_{V^h}^2$. 
We present the results in Table \ref{TabLapL2Q2}. We see 
from this results that the error in the seminorm 
$| \cdot |_{V^h}$ decays linearly with $h$ and recall that this seminorm is scaled by a factor $h$ in the definition 
of the extended norm $\| \cdot \|_{V^h}$ in \eqref{eq:norm-for-P2}.

\begin{table}[ht]
\centering
\centering
\renewcommand{\arraystretch}{1.1}
     	\begin{tabular}{c|c} 
       	    $M$      & $| p-p^h |_{V^h}$  \\
			\hline  $1$  & $ 3.6040\times 10^{0} $\\
			\hline  $2$  &   $1.8127\times 10^{0}$\\
			\hline  $3$  &  $  9.0885\times 10^{-1}$\\
			\hline  $4$  &  $ 4.5506\times 10^{-1}$\\
			\hline  $5$  &  $ 2.2769\times 10^{-1}$\\
			\hline  $6$  &   $1.1388\times 10^{-1}$\\
			\hline  $7$  &   $5.6954\times 10^{-3}$\\
			\hline  $8$  &   $2.8480\times 10^{-3}$\\
			\hline  $9$  &  $ 1.4240\times 10^{-3}$\\
	    \end{tabular}
\caption{Table of scaled seminorm errors, see 
\eqref{def:seminormVh}, for \textbf{FV} solution,  $h = 2^{-M}$.
Recall that the seminorm $| \cdot |_{V^h}$  in \eqref{def:seminormVh} is scaled by a factor $h$ in the definition 
of the extended norm \eqref{eq:norm-for-P2}}. 
\label{TabLapL2Q2}
\end{table}

Using our high order formulation we compute the conservative 
approximation of the pressure and a Lagrange multiplier which 
is used to correct the solution for a improved $L^2$ 
approximation. Note that the exact solution value  of the 
Lagrange multiplier is $\lambda=0$. We now compute the error 
in the Lagrange multiplier approximation in the $M_h$ norm. The results are 
presented in Table \ref{TabMULQ2}. We observe a convergence 
of order $2$ in the approximation of the Lagrange multiplier. 

\begin{table}[ht]
\centering
\centering
\renewcommand{\arraystretch}{1.1}
     	\begin{tabular}{c|c} 
       	    $M$      & Error  \\
			\hline  $1$  &  $2.4825 \times 10^{-1} $\\
			\hline  $2$  &  $9.9023\times 10^{-2}$\\
			\hline  $3$  &  $ 2.5293\times 10^{-2}$\\
			\hline  $4$  &  $ 6.3369\times 10^{-3}$\\
			\hline  $5$  &  $ 1.5848\times 10^{-3}$\\
			\hline  $6$  &  $ 3.9623\times 10^{-4}$\\
			\hline  $7$  &  $ 9.9061\times 10^{-5}$\\
			\hline  $8$  &  $2.4765\times 10^{-5}$\\
			\hline  $9$  &  $ 6.1913\times 10^{-5}$\\
	    \end{tabular}
        \caption{Table of error values $\|\lambda_h-\lambda\|_{M_h}$ for the Lagrange multiplier approximation.} 
\label{TabMULQ2}
\end{table}

To finish this subsection we compute energy and conservation 
of mass indicators in Table \ref{tab:conservationindicator}.
The energy is defined as\begin{equation}
E(p) = \frac{1}{2}\int_\Omega  
\vert \nabla p\vert^2 dx - \int_\Omega qp 
\end{equation}
while the conservation of mass indicator is given by, 
\begin{equation}
J(p) = \left(\sum_{R} \left(\int_{\partial R}  -
\nabla p \cdot \mathbf{n}  - \int_{R} q \right)^2\right)^{1/2}.
\end{equation}

\begin{table}[ht]
\centering
\centering
\renewcommand{\arraystretch}{1.1}
     	\begin{tabular}{c|c|c||c|c} 
       	    $M$      &  $\mathbb{Q}^1, E(u_{FEM})$   &  $\mathbb{Q}^2, E(u_{FEM})$ &  $\mathbb{Q}^1, E(u_{FV})$   &  $\mathbb{Q}^2,  E(u_{FV})$    \\
			\hline  $1$  &  -4.5230278474  &  -4.523568683883 &  -4.5230278425  &  -4.5233568683864 \\ \hline
       	    $M$      &  $\mathbb{Q}^1, J(u_{FEM})$   &  $\mathbb{Q}^2, J(u_{FEM})$ &  $\mathbb{Q}^1, J(u_{FV})$   &  $\mathbb{Q}^2,  J(u_{FV})$    \\
			\hline  $1$  &  $5.2434\times 10^{-6}$  &  $8.2205\times 10^{-8}$ &  $2.2928\times 10^{-14}$  &  $1.0261\times 10^{-13}$ \\
			\end{tabular}
\caption{Energy minimization and conservation indicator with $h = 2^{-9}$.} 
\label{tab:conservationindicator}
\end{table}

\subsection{Singular forcing with nonhomogeneous Neumann boundary condition}

For comparison, we also solve two problems with Neumann 
boundary conditions. The first problem has a singular 
forcing term in the form of a font located at $(0,0)$ 
and a source located in $(1,1)$. The computed solution 
for this problem is shown in the Figure \ref{fig:fivespot}. 
The second problem has a smooth forcing term.

Table \ref{TabH12} shows $FEM$ and $FV$ computed order of 
convergence of the error. Apart from computing $L^1$ and 
$L^2$ norms of the error we also include the measure of 
the error in the seminorm $W^{1,1}$ (note that in this 
case the solution of this  problems in not regular and is not
in $H^1(\Omega)$). We 
observe here that, in terms of approximation, the performance 
of both strategies FEM and FV perform similarly with respect 
to the order of the polynomials. The main difference between 
the two computed solution is only the conservation of 
mass that is being satisfied only by the FV solution.

\begin{figure}[ht]
\centering
\includegraphics[scale=.2]{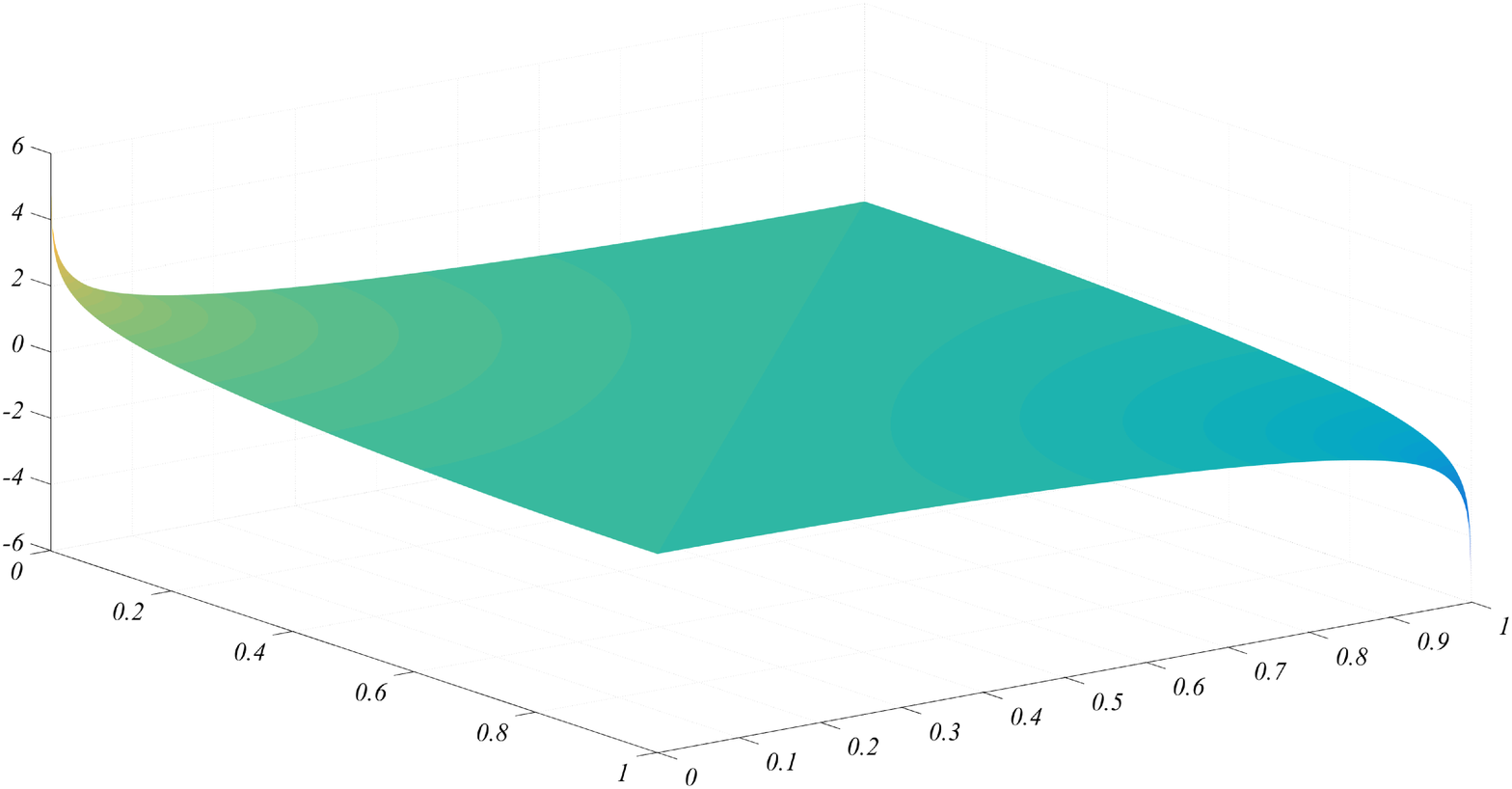} 
\caption{Plot of numerical solution for the problem 
with homogeneous Neumann boundary conditions and
singular right hand side.}\label{fig:fivespot}\end{figure}

\begin{table}[ht]
\centering
\renewcommand{\arraystretch}{1.1}
     	\begin{tabular}{c|c|c|c} 
       	     $FEM$  &            & $\mathbb{Q}^1$          & $ \mathbb{Q}^2$\\ 
			\hline  &    $L^1$          &  $1.8463$  								 &  $1.8707$  \\
			\hline  &    $L^2$          &  $1.0000$                                  &  $1.0121$  \\
			\hline  &    $W^{1,1}$     &  $0.8694$  								 &  $0.9983$  \\
			\hhline{|=|=|=|=|}
       	$FV$  &             &           &  \\ 
			\hline $FV+\lambda$ &    $L^1$          &  $1.8490$  							    &  $1.8715$  \\
			\hline $FV+\lambda$ &    $L^2$          &  $1.0000$  								&  $1.0000$  \\
			\hline  &    $W^{1,1}$     &  $0.8590$  								&  $0.9977$  \\			
	    \end{tabular}
\caption{Values of $L^1$, $L^2$  and $W^{1,1}$  error order 
of $FEM$ and $FV$ for the homogeneous Neumann boundary 
condition problem with singular forcing.} 
\label{TabH12}
\end{table}

\subsubsection{Smooth forcing}

To finish our comparison with Neumann boundary condition 
we consider the case where the flux term is given by 
$q(x,y) = x - y$. In Table \ref{Tab:Regular} we show the 
results. We obtain expected results with our FV formulation 
being as accurate as the FEM formulation and still 
satisfying the conservation of mass restrictions. 

 \begin{table}[ht]
\centering
\centering
\renewcommand{\arraystretch}{1.1}
     	\begin{tabular}{c|c|c|c} 
       	     $FEM$  &           & $\mathcal{Q}^1$          & $\mathcal{Q}^2$\\ 
			\hline  &    $L^1$          &  $1.9999$  								 &  $3.0000$  \\
			\hline  &    $L^2$          &  $1.9999$                                  &  $3.0000$  \\
			\hline  &    $W^{1,1}$     &  $1.0000$  								 &  $2.0000$  \\
			\hhline{|=|=|=|=|}
       	$FV$  &             &           &  \\ 
			\hline $FV+\lambda$ &    $L^1$          &  $2.0000$  							    &  $3.0000$  \\
			\hline $FV+\lambda$ &    $L^2$          &  $1.9999$  								&  $3.0000$  \\
			\hline  &    $W^{1,1}$     &  $1.0000$  								&  $2.0000$  \\			
	    \end{tabular}
\caption{Values  of  $L^1$, $L^2$  and $W^{1,1}$  error order of 
$FEM$ and $FV$ for the homogeneous Neumann boundary condition 
problem with smooth forcing.} 
\label{Tab:Regular}
\end{table}

  \section{Conclusions}\label{sec:conclusions}

  In this paper, we introduce a high-order discretization with locally conservative properties for a second-order problem. Our formulation discretizes the second order problem and there is no need to write an equivalent first order system of  differential equations. It is, therefore, a novel approach and it is fundamentally different from classical mixed finite element methods such as discretizing by Raviart-Thomas elements. We impose the conservative constraints by using a Lagrange multiplier for each control volume and therefore we can compute locally conservative solutions while keeping the high-order approximation. For the case of constant permeability coefficient, we present the analysis of our formulation at the continuous and discrete levels. In particular, we obtain optimal estimates for the $H^1$ and $L^2$ norms. We mention also that the optimal $L^2$ approximation is obtained without any post-processing or hybridization which are other differences with classical mixed finite element methods. The analysis can be straightforwardly extended to the case of smooth permeability coefficients.

We present numerical experiments that verify our theoretical findings. We also stress the fact that our approximation of the solution has continuous tangential fluxes along primal element edges. The implementation of our method is simple and requires only coding tools used for classical conforming high-order finite element method plus the computation of fluxes of basis functions along control volumes boundaries (as in the classical low-order finite volume method).

Our formulation can be easily extended to a variety of cases where both high-order approximation and also conservative properties are desirable. For instance, we mention the case of flow problems in high-contrast multiscale porous media with sophisticated high-order discretization schemes, see \cite{MR3430146}. We note that the analysis for this case and for other high-order approximation spaces is non-trivial as well as robust solvers are under investigation.

\vspace{3mm}
\noindent {\textcolor{black}{\bfseries Acknowledgement:} \\
Eduardo Abreu thanks in part by FAPESP 2016/23374-1 and 
CNPq Universal 445758/2014-7. Ciro Diaz thanks CAPES for 
a graduate fellowship. Marcus Sarkis thanks in part by 
NSF-MRI 1337943 and NSF-MPS 1522663.}

\bibliographystyle{elsarticle-num}
\bibliography{references}

\end{document}